\documentclass[11pt]{amsart}
\usepackage{hyperref}
\usepackage{amsfonts}
\usepackage{amssymb}
\usepackage{amsmath}
\usepackage[pdftex]{graphicx}

\input{tcilatex}

\begin{document}
\title[INEQUALITIES FOR $m-$CONVEX FUNCTIONS]{NEW OSTROWSKI TYPE
INEQUALITIES FOR $m-$CONVEX FUNCTIONS AND APPLICATIONS}
\author{HAVVA KAVURMACI$^{\blacktriangle ,\bigstar }$}
\address{$^{\blacktriangle }$ATATURK UNIVERSITY, K.K. EDUCATION FACULTY,
DEPARTMENT OF MATHEMATICS, 25240, CAMPUS, ERZURUM, TURKEY}
\email{havva.kvrmc@yahoo.com}
\thanks{$^{\bigstar }$Corresponding Author}
\author{M. EM\.{I}N \"{O}ZDEM\.{I}R$^{\spadesuit }$}
\address{$^{\spadesuit }$GRADUATE SCHOOL OF NATURAL AND APPLIED SCIENCE, A%
\u{G}RI \.{I}BRAH\.{I}M \c{C}E\c{C}EN UNIVERSITY, A\u{G}RI, TURKEY}
\email{emos@atauni.edu.tr}
\author{MERVE AVCI$^{\blacktriangle }$}
\email{merveavci@ymail.com}

\begin{abstract}
In this paper, we establish new inequalities of Ostrowski type for functions
whose derivatives in absolute value are $m-$convex. We also give some
applications to special means of positive real numbers. Finally, we obtain
some error estimates for the midpoint formula.
\end{abstract}

\maketitle

\section{INTRODUCTION}

Let $f:I\subset \lbrack 0,\infty )\rightarrow 
%TCIMACRO{\U{211d} }%
%BeginExpansion
\mathbb{R}
%EndExpansion
$ be a differentiable mapping on $I^{\circ }$, the interior of the interval $%
I$, such that $f^{\prime }\in L\left( \left[ a,b\right] \right) $ where $a$ $%
,$ $b\in I$ with $a<b$ . If $\left\vert f^{\prime }\left( x\right)
\right\vert \leq M$ , then the following inequality holds (see \cite{ADDC}):

\begin{equation*}
\left\vert f(x)-\frac{1}{b-a}\int_{a}^{b}f(u)du\right\vert \leq \frac{M}{b-a}%
\left[ \frac{\left( x-a\right) ^{2}+\left( b-x\right) ^{2}}{2}\right] .
\end{equation*}

This inequality is well known in the literature as the \textit{Ostrowski
inequality.}\textbf{\ }For some results which generalize, improve, and
extend the above inequality, see \cite{ADDC},\cite{BCDPS},\cite{CDR},\cite%
{DS1} and \cite{DW}, the references therein.

In \cite{GT1}, G. Toader defined $m-$convexity, an intermediate between the
usual convexity and starshaped property, as the following:

\begin{definition}
\label{d.1.1} The function $f:[0,b]\rightarrow
%TCIMACRO{\U{211d} }%
%BeginExpansion
\mathbb{R}
%EndExpansion
,$ $b>0$ , is said to be $m-$convex, where $m\in \lbrack 0,1]$, if we have%
\begin{equation*}
f(tx+m(1-t)y)\leq tf(x)+m(1-t)f(y)
\end{equation*}%
for all $x,y\in \left[ 0,b\right] $ and $t\in \left[ 0,1\right] $.
\end{definition}

Denote by $K_{m}(b)$ the set of the $m-$convex functions on $\left[ 0,b%
\right] $ for which $f(0)\leq 0$.

\begin{definition}
\label{d.1.2} The function $f:[0,b]\rightarrow
%TCIMACRO{\U{211d} }%
%BeginExpansion
\mathbb{R}
%EndExpansion
,$ $b>0$ is said to be starshaped if for every $x\in \left[ 0,b\right] $ and
$t\in \lbrack 0,1]$ we have:%
\begin{equation*}
f(tx)\leq tf(x).
\end{equation*}
\end{definition}

For $m=1$, we recapture the concept of convex functions defined on $[0,b]$
and $m=0$ we get the concept of starshaped functions on $[0,b]$.

The following theorem contains the Hermite-Hadamard type integral inequality
(see \cite{DCK}).

\begin{theorem}
\label{t.1.1} Let $f:I\subset
%TCIMACRO{\U{211d} }%
%BeginExpansion
\mathbb{R}
%EndExpansion
\rightarrow
%TCIMACRO{\U{211d} }%
%BeginExpansion
\mathbb{R}
%EndExpansion
$ be an M-Lipschitzian mapping on $I$ and $a,b\in I$ with $a<b.$ Then we
have the inequality:
\end{theorem}

\begin{equation}
\left\vert f\left( \frac{a+b}{2}\right) -\frac{1}{b-a}\int_{a}^{b}f(x)dx%
\right\vert \leq M\frac{(b-a)}{4}.  \label{h.1.1}
\end{equation}

In \cite{SOS} E. Set, M.E. \"{O}zdemir, M.Z. Sar\i kaya established the
following theorem.

\begin{theorem}
\label{t.1.2} Let $f:I^{\circ }\subset \left[ 0,b^{\ast }\right] \rightarrow
%TCIMACRO{\U{211d} }%
%BeginExpansion
\mathbb{R}
%EndExpansion
,$ $b^{\ast }>0,$ be a differentiable mapping on $I^{\circ }$, $a,b\in $ $%
I^{\circ }$ with $a<b.$ If $\left\vert f^{\prime }\right\vert ^{q}$ is $m-$%
convex on $[a,b]$, $q>1$ and $m\in (0,1],$ then the following inequality
holds:
\end{theorem}

\begin{equation}
\left\vert f\left( \frac{a+b}{2}\right) -\frac{1}{b-a}\int_{a}^{b}f(x)dx%
\right\vert \leq \left( b-a\right) \left( \frac{3^{1-\left( \frac{1}{q}%
\right) }}{8}\right) \left( \left\vert f^{\prime }(a)\right\vert +m^{\frac{1%
}{q}}\left\vert f^{\prime }(\frac{b}{m})\right\vert \right) .  \label{h.1.2}
\end{equation}%
where $\frac{b}{m}<b^{\ast }.$

In \cite{UK} U. Kirmaci proved the following theorem.

\begin{theorem}
\label{t.1.3} Let $f:I^{\circ }\subset
%TCIMACRO{\U{211d} }%
%BeginExpansion
\mathbb{R}
%EndExpansion
\rightarrow
%TCIMACRO{\U{211d} }%
%BeginExpansion
\mathbb{R}
%EndExpansion
$ be a differentiable mapping on $I^{\circ }$, $a,b\in I^{\circ }$ with $a<b$%
. If the mapping $\left\vert f^{\prime }\right\vert $ is convex on $[a,b]$,
then we have
\end{theorem}

\begin{equation}
\left\vert f\left( \frac{a+b}{2}\right) -\frac{1}{b-a}\int_{a}^{b}f(x)dx%
\right\vert \leq \frac{b-a}{8}\left( \left\vert f^{\prime }(a)\right\vert
+\left\vert f^{\prime }(b)\right\vert \right) .  \label{h.1.3}
\end{equation}

In \cite{DG} S.S. Dragomir and G. Toader proved the following
Hermite-Hadamard type inequality for $m-$convex functions.

\begin{theorem}
\label{t.1.4} Let $f:[0,\infty )\rightarrow
%TCIMACRO{\U{211d} }%
%BeginExpansion
\mathbb{R}
%EndExpansion
$ be an $m-$convex function with $m\in (0,1]$. If $\ 0\leq a<b<\infty $ and $%
f\in L^{1}([a,b])$ then
\end{theorem}

\begin{equation}
\frac{1}{b-a}\int_{a}^{b}f(x)dx\leq \min \left\{ \frac{f(a)+mf(\frac{b}{m})}{%
2},\frac{f(b)+mf(\frac{a}{m})}{2}\right\} .  \label{h.1.4}
\end{equation}

Some generalizations of this result can be found in \cite{BPR}.

In \cite{BOP} M.K. Bakula, M.E. \"{O}zdemir and J. Pe\v{c}ari\'{c} proved
the following theorems.

\begin{theorem}
\label{t.1.5} Let $I$ be an open real interval such that $[0,\infty )\subset
I$. Let $f:I\rightarrow
%TCIMACRO{\U{211d} }%
%BeginExpansion
\mathbb{R}
%EndExpansion
$ be a differentiable function on $I$ such that $f^{\prime }\in L([a,b])$,
where $0\leq a<b<\infty $ . If $\left\vert f^{\prime }\right\vert ^{q}$ is $%
m-$convex on $[a,b]$ for some fixed $m\in (0,1]$ and $q\in \lbrack 1,\infty
) $ , then
\end{theorem}

\begin{equation}
\left\vert \frac{f(a)+f(b)}{2}-\frac{1}{b-a}\int_{a}^{b}f(x)dx\right\vert
\leq \frac{b-a}{4}\left( \mu _{1}^{\frac{1}{q}}+\mu _{2}^{\frac{1}{q}%
}\right) ,  \label{h.1.5}
\end{equation}%
where

\begin{eqnarray*}
\mu _{1} &=&\min \left\{ \frac{\left\vert f^{\prime }(a)\right\vert
^{q}+m\left\vert f^{\prime }(\frac{a+b}{2m})\right\vert ^{q}}{2},\frac{%
\left\vert f^{\prime }(\frac{a+b}{2})\right\vert ^{q}+m\left\vert f^{\prime
}(\frac{a}{m})\right\vert ^{q}}{2}\right\} , \\
\mu _{2} &=&\min \left\{ \frac{\left\vert f^{\prime }(b)\right\vert
^{q}+m\left\vert f^{\prime }(\frac{a+b}{2m})\right\vert ^{q}}{2},\frac{%
\left\vert f^{\prime }(\frac{a+b}{2})\right\vert ^{q}+m\left\vert f^{\prime
}(\frac{b}{m})\right\vert ^{q}}{2}\right\} .
\end{eqnarray*}

\begin{theorem}
\label{t.1.6} Let $I$ be an open real interval such that $[0,\infty )\subset
I$. Let $f:I\rightarrow
%TCIMACRO{\U{211d} }%
%BeginExpansion
\mathbb{R}
%EndExpansion
$ be a differentiable function on $I$ such that $f^{\prime }\in L([a,b])$,
where $0\leq a<b<\infty $ . If $\left\vert f^{\prime }\right\vert ^{q}$ is $%
m-$convex on $[a,b]$ for some fixed $m\in (0,1]$ and $q\in \lbrack 1,\infty
) $ , then%
\begin{eqnarray}
&&\left\vert f\left( \frac{a+b}{2}\right) -\frac{1}{b-a}\int_{a}^{b}f(x)dx%
\right\vert  \label{h.1.6} \\
&\leq &\frac{b-a}{4}\min \left\{ \left( \frac{\left\vert f^{\prime
}(a)\right\vert ^{q}+m\left\vert f^{\prime }(\frac{b}{m})\right\vert ^{q}}{2}%
\right) ^{\frac{1}{q}},\left( \frac{m\left\vert f^{\prime }(\frac{a}{m}%
)\right\vert ^{q}+\left\vert f^{\prime }(b)\right\vert ^{q}}{2}\right) ^{%
\frac{1}{q}}\right\} .  \notag
\end{eqnarray}
\end{theorem}

The main purpose of this paper is to establish new Ostrowski type
inequalities for functions whose derivatives in absolute value are $m-$%
convex. Using these results we give some applications to special means of
positive real numbers and we obtain some error estimates for the midpoint
formula.

\section{THE RESULTS}

In \cite{AD}, in order to prove some inequalities related to Ostrowski
inequality, M. Alomari and M. Darus used the following lemma with the
constant $(b-a),$ but we changed it with the constant $(a-b)$ to obtain an
equality in Lemma \ref{l.2.1}.

\begin{lemma}
\label{l.2.1} Let $f:I\subset
%TCIMACRO{\U{211d} }%
%BeginExpansion
\mathbb{R}
%EndExpansion
\rightarrow
%TCIMACRO{\U{211d} }%
%BeginExpansion
\mathbb{R}
%EndExpansion
$ be a differentiable mapping on $I^{\circ }$ where $a,b\in I$ with $a<b.$
If $f^{\prime }\in L\left( [a,b]\right) $, then the following equality holds:
\end{lemma}

\begin{equation}
f(x)-\frac{1}{b-a}\int_{a}^{b}f(u)du=(a-b)\int_{0}^{1}p(t)f^{\prime
}(ta+(1-t)b)dt  \label{m.2.1}
\end{equation}%
for each $t\in \left[ 0,1\right] $ , where

\begin{equation*}
p(t)=\left\{ 
\begin{array}{ccc}
t & , & t\in \left[ 0,\frac{b-x}{b-a}\right] \\ 
&  &  \\ 
t-1 & , & t\in \left( \frac{b-x}{b-a},1\right]%
\end{array}%
\right. ,
\end{equation*}%
for all $x\in \left[ a,b\right] .$

\begin{theorem}
\label{t.2.1} Let $I$ be an open real interval such that $[0,\infty )\subset
I$. Let $f:I\rightarrow
%TCIMACRO{\U{211d} }%
%BeginExpansion
\mathbb{R}
%EndExpansion
$ be a differentiable function on $I$ such that $f^{\prime }\in L([a,b])$,
where $0\leq a<b<\infty $ . If $\left\vert f^{\prime }\right\vert $ is $m-$%
convex on $[a,b]$ for some fixed $m\in (0,1]$, then the following inequality
holds:
\end{theorem}

\begin{align}
& \left\vert f(x)-\frac{1}{b-a}\int_{a}^{b}f(u)du\right\vert  \label{m.2.2}
\\
\leq (b-a)& \min \left\{ 
\begin{array}{c}
\left[ \frac{1}{6}-\frac{1}{2}\left( \frac{b-x}{b-a}\right) ^{2}+\frac{2}{3}%
\left( \frac{b-x}{b-a}\right) ^{3}\right] \left\vert f^{\prime
}(a)\right\vert \\ 
+m\left[ \frac{1}{2}\left( \frac{b-x}{b-a}\right) ^{2}-\frac{1}{3}\left( 
\frac{b-x}{b-a}\right) ^{3}+\frac{1}{3}\left( \frac{x-a}{b-a}\right) ^{3}%
\right] \left\vert f^{\prime }\left( \frac{b}{m}\right) \right\vert%
\end{array}%
,\right.  \notag \\
& \left. 
\begin{array}{c}
\left[ \frac{1}{6}-\frac{1}{2}\left( \frac{b-x}{b-a}\right) ^{2}+\frac{2}{3}%
\left( \frac{b-x}{b-a}\right) ^{3}\right] \left\vert f^{\prime
}(b)\right\vert \\ 
+m\left[ \frac{1}{2}\left( \frac{b-x}{b-a}\right) ^{2}-\frac{1}{3}\left( 
\frac{b-x}{b-a}\right) ^{3}+\frac{1}{3}\left( \frac{x-a}{b-a}\right) ^{3}%
\right] \left\vert f^{\prime }\left( \frac{a}{m}\right) \right\vert%
\end{array}%
\right\} .  \notag
\end{align}%
for each $x\in \left[ a,b\right] .$

\begin{proof}
By Lemma \ref{l.2.1}, we have%
\begin{eqnarray*}
&&\left\vert f(x)-\frac{1}{b-a}\int_{a}^{b}f(u)du\right\vert \\
&\leq &(b-a)\int_{0}^{\frac{b-x}{b-a}}t\left\vert f^{\prime
}(ta+(1-t)b)\right\vert dt \\
&&+(b-a)\int_{\frac{b-x}{b-a}}^{^{1}}(1-t)\left\vert f^{\prime
}(ta+(1-t)b)\right\vert dt
\end{eqnarray*}%
Since $\left\vert f^{\prime }\right\vert $ is $m-$convex on $\left[ a,b%
\right] $ we know that for any $t\in \left[ 0,1\right] $%
\begin{eqnarray*}
\left\vert f^{\prime }(ta+(1-t)b)\right\vert &=&\left\vert f^{\prime
}(ta+m(1-t)\frac{b}{m})\right\vert \\
&\leq &t\left\vert f^{\prime }(a)\right\vert +m(1-t)\left\vert f^{\prime
}\left( \frac{b}{m}\right) \right\vert ,
\end{eqnarray*}%
Hence%
\begin{eqnarray*}
&&\left\vert f(x)-\frac{1}{b-a}\int_{a}^{b}f(u)du\right\vert \\
&\leq &(b-a)\int_{0}^{^{\frac{b-x}{b-a}}}t\left[ t\left\vert f^{\prime
}(a)\right\vert +m(1-t)\left\vert f^{\prime }\left( \frac{b}{m}\right)
\right\vert \right] dt \\
&&+(b-a)\int_{\frac{b-x}{b-a}}^{^{1}}(1-t)\left[ t\left\vert f^{\prime
}(a)\right\vert +m(1-t)\left\vert f^{\prime }\left( \frac{b}{m}\right)
\right\vert \right] dt \\
&=&(b-a)\left\{ 
\begin{array}{c}
\left[ \frac{1}{6}-\frac{1}{2}\left( \frac{b-x}{b-a}\right) ^{2}+\frac{2}{3}%
\left( \frac{b-x}{b-a}\right) ^{3}\right] \left\vert f^{\prime
}(a)\right\vert \\ 
+m\left[ \frac{1}{2}\left( \frac{b-x}{b-a}\right) ^{2}-\frac{1}{3}\left( 
\frac{b-x}{b-a}\right) ^{3}+\frac{1}{3}\left( \frac{x-a}{b-a}\right) ^{3}%
\right] \left\vert f^{\prime }\left( \frac{b}{m}\right) \right\vert%
\end{array}%
\right\}
\end{eqnarray*}%
where we use the facts that%
\begin{eqnarray*}
&&\int_{0}^{^{\frac{b-x}{b-a}}}t\left[ t\left\vert f^{\prime }(a)\right\vert
+m(1-t)\left\vert f^{\prime }\left( \frac{b}{m}\right) \right\vert \right] dt
\\
&=&\frac{1}{3}\left( \frac{b-x}{b-a}\right) ^{3}\left\vert f^{\prime
}(a)\right\vert +m\left[ \frac{1}{2}\left( \frac{b-x}{b-a}\right) ^{2}-\frac{%
1}{3}\left( \frac{b-x}{b-a}\right) ^{3}\right] \left\vert f^{\prime }\left( 
\frac{b}{m}\right) \right\vert ,
\end{eqnarray*}%
and%
\begin{eqnarray*}
&&\int_{\frac{b-x}{b-a}}^{^{1}}(1-t)\left[ t\left\vert f^{\prime
}(a)\right\vert +m(1-t)\left\vert f^{\prime }\left( \frac{b}{m}\right)
\right\vert \right] dt \\
&=&\left[ \frac{1}{6}-\frac{1}{2}\left( \frac{b-x}{b-a}\right) ^{2}+\frac{1}{%
3}\left( \frac{b-x}{b-a}\right) ^{3}\right] \left\vert f^{\prime
}(a)\right\vert +m\frac{1}{3}\left( \frac{x-a}{b-a}\right) ^{3}\left\vert
f^{\prime }\left( \frac{b}{m}\right) \right\vert ,
\end{eqnarray*}%
and analogously%
\begin{eqnarray*}
&&\left\vert f(x)-\frac{1}{b-a}\int_{a}^{b}f(u)du\right\vert \\
&\leq &(b-a)\left\{ 
\begin{array}{c}
\left[ \frac{1}{6}-\frac{1}{2}\left( \frac{b-x}{b-a}\right) ^{2}+\frac{2}{3}%
\left( \frac{b-x}{b-a}\right) ^{3}\right] \left\vert f^{\prime
}(b)\right\vert \\ 
+m\left[ \frac{1}{2}\left( \frac{b-x}{b-a}\right) ^{2}-\frac{1}{3}\left( 
\frac{b-x}{b-a}\right) ^{3}+\frac{1}{3}\left( \frac{x-a}{b-a}\right) ^{3}%
\right] \left\vert f^{\prime }\left( \frac{a}{m}\right) \right\vert%
\end{array}%
\right\} .
\end{eqnarray*}%
The proof is completed.
\end{proof}

\begin{remark}
\label{r.2.1} Suppose that all the assumptions of Theorem \ref{t.2.1} are
satisfied. If we choose $x=\frac{a+b}{2}$, then we have%
\begin{eqnarray*}
&&\left\vert f\left( \frac{a+b}{2}\right) -\frac{1}{b-a}\int_{a}^{b}f(u)du%
\right\vert \\
&\leq &\frac{b-a}{8}\min \left\{ \left\vert f^{\prime }(a)\right\vert
+m\left\vert f^{\prime }\left( \frac{b}{m}\right) \right\vert ,\left\vert
f^{\prime }(b)\right\vert +m\left\vert f^{\prime }\left( \frac{a}{m}\right)
\right\vert \right\}
\end{eqnarray*}%
which is [(\ref{h.1.6}),q=1].
\end{remark}

\begin{remark}
\label{r.2.2} Suppose that all the assumptions of Theorem \ref{t.2.1} are
satisfied. Then
\end{remark}

$(A)$ If we choose $m=1$ and $x=\frac{a+b}{2}$, we obtain

\begin{equation*}
\left\vert f\left( \frac{a+b}{2}\right) -\frac{1}{b-a}\int_{a}^{b}f(u)du%
\right\vert \leq \frac{b-a}{8}\left( \left\vert f^{\prime }(a)\right\vert
+\left\vert f^{\prime }(b)\right\vert \right) ,
\end{equation*}%
which is (\ref{h.1.3}).

$(B)$ In $(A)$. Additionally, if we choose $\left\vert f^{\prime
}(x)\right\vert \leq M,$ $M>0$

\begin{equation*}
\left\vert f\left( \frac{a+b}{2}\right) -\frac{1}{b-a}\int_{a}^{b}f(u)du%
\right\vert \leq M\frac{(b-a)}{4}
\end{equation*}%
which is (\ref{h.1.1}).

\begin{theorem}
\label{t.2.2} Let $I$ be an open real interval such that $[0,\infty )\subset
I$. Let $f:I\rightarrow
%TCIMACRO{\U{211d} }%
%BeginExpansion
\mathbb{R}
%EndExpansion
$ be a differentiable function on $I$ such that $f^{\prime }\in L([a,b])$,
where $0\leq a<b<\infty $ . If $\left\vert f^{\prime }\right\vert ^\frac{p}{p-1}$ is $%
m-$convex on $[a,b]$ for some fixed $m\in (0,1]$ and $p>1,$ $\frac{1}{p}+%
\frac{1}{q}=1$ , then the following inequality holds:%
\begin{eqnarray}
&&\left\vert f(x)-\frac{1}{b-a}\int_{a}^{b}f(u)du\right\vert  \label{m.2.3}
\\
\leq \frac{1}{\left( p+1\right) ^{\frac{1}{p}}} &&\left\{
\begin{array}{c}
\frac{(b-x)^{2}}{b-a}\left[ \min \left\{ \frac{\left\vert f^{\prime
}(b)\right\vert ^{q}+m\left\vert f^{\prime }\left( \frac{x}{m}\right)
\right\vert ^{q}}{2},\frac{\left\vert f^{\prime }(x)\right\vert
^{q}+m\left\vert f^{\prime }\left( \frac{b}{m}\right) \right\vert ^{q}}{2}%
\right\} \right] ^{\frac{1}{q}} \\
+\frac{(x-a)^{2}}{b-a}\left[ \min \left\{ \frac{\left\vert f^{\prime
}(a)\right\vert ^{q}+m\left\vert f^{\prime }\left( \frac{x}{m}\right)
\right\vert ^{q}}{2},\frac{\left\vert f^{\prime }(x)\right\vert
^{q}+m\left\vert f^{\prime }\left( \frac{a}{m}\right) \right\vert ^{q}}{2}%
\right\} \right] ^{\frac{1}{q}}%
\end{array}%
\right\}  \notag
\end{eqnarray}%
for each $x\in \left[ a,b\right] .$
\end{theorem}

\begin{proof}
From Lemma \ref{l.2.1} and using the H\"{o}lder inequality, we have%
\begin{eqnarray*}
&&\left\vert f(x)-\frac{1}{b-a}\int_{a}^{b}f(u)du\right\vert \\
&\leq &(b-a)\left( \int_{0}^{\frac{b-x}{b-a}}t^{p}dt\right) ^{\frac{1}{p}%
}\left( \int_{0}^{\frac{b-x}{b-a}}\left\vert f^{\prime }\left(
ta+(1-t)b\right) \right\vert ^{q}dt\right) ^{\frac{1}{q}} \\
&&+(b-a)\left( \int_{\frac{b-x}{b-a}}^{1}(1-t)^{p}dt\right) ^{\frac{1}{p}%
}\left( \int_{\frac{b-x}{b-a}}^{1}\left\vert f^{\prime }\left(
ta+(1-t)b\right) \right\vert ^{q}dt\right) ^{\frac{1}{q}} \\
&\leq &(b-a)\left( \frac{b-x}{b-a}\right) ^{\frac{p+1}{p}}\left( \frac{1}{p+1%
}\right) ^{\frac{1}{p}}\left( \frac{b-x}{b-a}\right) ^{\frac{1}{q}} \\
&&\times \left( \min \left\{ \frac{\left\vert f^{\prime }(b)\right\vert
^{q}+m\left\vert f^{\prime }\left( \frac{x}{m}\right) \right\vert ^{q}}{2},%
\frac{\left\vert f^{\prime }(x)\right\vert ^{q}+m\left\vert f^{\prime
}\left( \frac{b}{m}\right) \right\vert ^{q}}{2}\right\} \right) ^{\frac{1}{q}%
} \\
&&+(b-a)\left( \frac{x-a}{b-a}\right) ^{\frac{p+1}{p}}\left( \frac{1}{p+1}%
\right) ^{\frac{1}{p}}\left( \frac{x-a}{b-a}\right) ^{\frac{1}{q}} \\
&&\times \left( \min \left\{ \frac{\left\vert f^{\prime }(a)\right\vert
^{q}+m\left\vert f^{\prime }\left( \frac{x}{m}\right) \right\vert ^{q}}{2},%
\frac{\left\vert f^{\prime }(x)\right\vert ^{q}+m\left\vert f^{\prime
}\left( \frac{a}{m}\right) \right\vert ^{q}}{2}\right\} \right) ^{\frac{1}{q}%
} \\
&=&\frac{1}{(p+1)^{\frac{1}{p}}}\frac{1}{b-a}\left\{ 
\begin{array}{c}
(b-x)^{2}\left[ \min \left\{ \frac{\left\vert f^{\prime }(b)\right\vert
^{q}+m\left\vert f^{\prime }\left( \frac{x}{m}\right) \right\vert ^{q}}{2},%
\frac{\left\vert f^{\prime }(x)\right\vert ^{q}+m\left\vert f^{\prime
}\left( \frac{b}{m}\right) \right\vert ^{q}}{2}\right\} \right] ^{\frac{1}{q}%
} \\ 
+(x-a)^{2}\left[ \min \left\{ \frac{\left\vert f^{\prime }(a)\right\vert
^{q}+m\left\vert f^{\prime }\left( \frac{x}{m}\right) \right\vert ^{q}}{2},%
\frac{\left\vert f^{\prime }(x)\right\vert ^{q}+m\left\vert f^{\prime
}\left( \frac{a}{m}\right) \right\vert ^{q}}{2}\right\} \right] ^{\frac{1}{q}%
}%
\end{array}%
\right\}
\end{eqnarray*}%
where we use the facts that%
\begin{eqnarray*}
\int_{0}^{\frac{b-x}{b-a}}t^{p}dt &=&\left( \frac{b-x}{b-a}\right) ^{p+1}%
\frac{1}{p+1}, \\
\int_{\frac{b-x}{b-a}}^{1}(1-t)^{p}dt &=&\left( \frac{x-a}{b-a}\right) ^{p+1}%
\frac{1}{p+1},
\end{eqnarray*}%
and by Theorem \ref{t.1.4} we get 
\begin{eqnarray*}
&&\frac{b-a}{b-x}\int_{0}^{\frac{b-x}{b-a}}\left\vert f^{\prime
}(ta+(1-t)b)\right\vert ^{q}dt \\
&\leq &\min \left\{ \frac{\left\vert f^{\prime }(b)\right\vert
^{q}+m\left\vert f^{\prime }\left( \frac{x}{m}\right) \right\vert ^{q}}{2},%
\frac{\left\vert f^{\prime }(x)\right\vert ^{q}+m\left\vert f^{\prime
}\left( \frac{b}{m}\right) \right\vert ^{q}}{2}\right\} , \\
&&\frac{b-a}{x-a}\int_{\frac{b-x}{b-a}}^{1}\left\vert f^{\prime
}(ta+(1-t)b)\right\vert ^{q}dt \\
&\leq &\min \left\{ \frac{\left\vert f^{\prime }(a)\right\vert
^{q}+m\left\vert f^{\prime }\left( \frac{x}{m}\right) \right\vert ^{q}}{2},%
\frac{\left\vert f^{\prime }(x)\right\vert ^{q}+m\left\vert f^{\prime
}\left( \frac{a}{m}\right) \right\vert ^{q}}{2}\right\} .
\end{eqnarray*}%
The proof is completed.
\end{proof}

\begin{corollary}
\label{c.2.1} Suppose that all the assumptions of Theorem \ref{t.2.2} are
satisfied, if we choose $\left\vert f^{\prime }(x)\right\vert \leq M$, $M>0$%
, then we have%
\begin{eqnarray*}
&&\left\vert f(x)-\frac{1}{b-a}\int_{a}^{b}f(u)du\right\vert \\
&\leq &\left( \frac{1}{\left( p+1\right) ^{\frac{1}{p}}}\right) \left( \frac{%
1+m}{2}\right) ^{\frac{1}{q}}M\left[ \frac{(b-x)^{2}+(x-a)^{2}}{b-a}\right] .
\end{eqnarray*}
\end{corollary}

\begin{corollary}
\label{c.2.2} Suppose that all the assumptions of Theorem \ref{t.2.2} are
satisfied, if we choose $x=\frac{a+b}{2}$ and $\frac{1}{2}<\left(
\frac{1}{p+1}\right) ^{\frac{1}{p}}<1$, then we have
\end{corollary}

\begin{equation*}
\left\vert f\left( \frac{a+b}{2}\right) -\frac{1}{b-a}\int_{a}^{b}f(u)du%
\right\vert \leq \frac{b-a}{4}\left( \mu _{1}^{\frac{1}{q}}+\mu _{2}^{\frac{1%
}{q}}\right) ,
\end{equation*}%
\textit{where}

\begin{eqnarray*}
\mu _{1} &=&\min \left\{ \frac{\left\vert f^{\prime }(b)\right\vert
^{q}+m\left\vert f^{\prime }(\frac{a+b}{2m})\right\vert ^{q}}{2},\frac{%
\left\vert f^{\prime }(\frac{a+b}{2})\right\vert ^{q}+m\left\vert f^{\prime
}\left( \frac{b}{m}\right) \right\vert ^{q}}{2}\right\} , \\
\mu _{2} &=&\min \left\{ \frac{\left\vert f^{\prime }(a)\right\vert
^{q}+m\left\vert f^{\prime }(\frac{a+b}{2m})\right\vert ^{q}}{2},\frac{%
\left\vert f^{\prime }(\frac{a+b}{2})\right\vert ^{q}+m\left\vert f^{\prime
}(\frac{a}{m})\right\vert ^{q}}{2}\right\} .
\end{eqnarray*}

\begin{remark}
\label{r.2.3} Corollary \ref{c.2.2} is similar to (\ref{h.1.5}) inequality,
but for the left-hand side of Hermite-Hadamard inequality.
\end{remark}

\begin{remark}
\label{r.2.4} Suppose that all the assumptions of Theorem \ref{t.2.2} are
satisfied. Then in Corollary \ref{c.2.2}
\end{remark}

$\left( D\right) $ $\left\vert f^{\prime }\right\vert $ is increasing and $%
m=1$ then we have

\begin{equation*}
\left\vert f\left( \frac{a+b}{2}\right) -\frac{1}{b-a}\int_{a}^{b}f(u)du%
\right\vert \leq \frac{b-a}{2}\left\vert f^{\prime }(b)\right\vert ,
\end{equation*}

$(E)$ $\left\vert f^{\prime }\right\vert $ is decreasing and $m=1$ then we
have

\begin{equation*}
\left\vert f\left( \frac{a+b}{2}\right) -\frac{1}{b-a}\int_{a}^{b}f(u)du%
\right\vert \leq \frac{b-a}{2}\left\vert f^{\prime }(a)\right\vert ,
\end{equation*}

$\left( F\right) $ $\left\vert f^{\prime }(b)\right\vert =\left\vert
f^{\prime }(a)\right\vert =\left\vert f^{\prime }\left( \frac{a+b}{2}\right)
\right\vert $ and $m=1$ then we have

\begin{equation*}
\left\vert f\left( \frac{a+b}{2}\right) -\frac{1}{b-a}\int_{a}^{b}f(u)du%
\right\vert \leq \frac{b-a}{2}\left\vert f^{\prime }\left( \frac{a+b}{2}%
\right) \right\vert .
\end{equation*}

\begin{theorem}
\label{t.2.3} Let $I$ be an open real interval such that $[0,\infty )\subset
I$. Let $f:I\rightarrow
%TCIMACRO{\U{211d} }%
%BeginExpansion
\mathbb{R}
%EndExpansion
$ be a differentiable function on $I$ such that $f^{\prime }\in L([a,b])$,
where $0\leq a<b<\infty $. If $\left\vert f^{\prime }\right\vert ^{q}$ is $%
m- $convex on $[a,b]$ for some fixed $m\in (0,1]$ and $q\in \lbrack 1,\infty
)$, $x\in \left[ a,b\right] $, then the following inequality holds:
\end{theorem}

\begin{eqnarray}
&&\left\vert f(x)-\frac{1}{b-a}\int_{a}^{b}f(u)du\right\vert  \label{m.2.4}
\\
&\leq &(b-a)\left( \frac{1}{2}\right) ^{1-\frac{1}{q}}  \notag \\
&&\left\{ 
\begin{array}{c}
\left( \frac{b-x}{b-a}\right) ^{2\left( 1-\frac{1}{q}\right) }\left[ \frac{1%
}{3}\left( \frac{b-x}{b-a}\right) ^{3}\left\vert f^{\prime }(a)\right\vert
^{q}+m\frac{(b-x)^{2}(b-3a+2x)}{6(b-a)^{3}}\left\vert f^{\prime }\left( 
\frac{b}{m}\right) \right\vert ^{q}\right] ^{\frac{1}{q}} \\ 
+\left( \frac{x-a}{b-a}\right) ^{2\left( 1-\frac{1}{q}\right) }\left[ \left( 
\frac{1}{6}+\frac{(b-x)^{2}(3a-b-2x)}{6(b-a)^{3}}\right) \left\vert
f^{\prime }(a)\right\vert ^{q}+m\frac{1}{3}\left( \frac{x-a}{b-a}\right)
^{3}\left\vert f^{\prime }\left( \frac{b}{m}\right) \right\vert ^{q}\right]
^{\frac{1}{q}}%
\end{array}%
\right\}  \notag
\end{eqnarray}%
\textit{for each }$x\in \left[ a,b\right] .$

\begin{proof}
By Lemma \ref{l.2.1} and using the well known power mean inequality we have%
\begin{eqnarray*}
&&\left\vert f(x)-\frac{1}{b-a}\int_{a}^{b}f(u)du\right\vert \\
&\leq &(b-a)\int_{0}^{^{\frac{b-x}{b-a}}}t\left\vert f^{\prime
}(ta+(1-t)b)\right\vert dt \\
&&+(b-a)\int_{\frac{b-x}{b-a}}^{^{1}}(1-t)\left\vert f^{\prime
}(ta+(1-t)b)\right\vert dt \\
&\leq &(b-a)\left( \int_{0}^{^{\frac{b-x}{b-a}}}tdt\right) ^{1-\frac{1}{q}%
}\left( \int_{0}^{^{\frac{b-x}{b-a}}}t\left\vert f^{\prime
}(ta+(1-t)b)\right\vert ^{q}dt\right) ^{\frac{1}{q}} \\
&&+(b-a)\left( \int_{\frac{b-x}{b-a}}^{^{1}}(1-t)dt\right) ^{1-\frac{1}{q}%
}\left( \int_{\frac{b-x}{b-a}}^{^{1}}(1-t)\left\vert f^{\prime
}(ta+(1-t)b)\right\vert ^{q}dt\right) ^{\frac{1}{q}} \\
&\leq &(b-a)\left( \frac{1}{2}\right) ^{1-\frac{1}{q}} \\
&&\left\{ 
\begin{array}{c}
\left( \frac{b-x}{b-a}\right) ^{2\left( 1-\frac{1}{q}\right) }\left[ \frac{1%
}{3}\left( \frac{b-x}{b-a}\right) ^{3}\left\vert f^{\prime }(a)\right\vert
^{q}+m\frac{(b-x)^{2}(b-3a+2x)}{6(b-a)^{3}}\left\vert f^{\prime }\left( 
\frac{b}{m}\right) \right\vert ^{q}\right] ^{\frac{1}{q}} \\ 
+\left( \frac{x-a}{b-a}\right) ^{2\left( 1-\frac{1}{q}\right) }\left[ \left( 
\frac{1}{6}+\frac{(b-x)^{2}(3a-b-2x)}{6(b-a)^{3}}\right) \left\vert
f^{\prime }(a)\right\vert ^{q}+m\frac{1}{3}\left( \frac{x-a}{b-a}\right)
^{3}\left\vert f^{\prime }\left( \frac{b}{m}\right) \right\vert ^{q}\right]
^{\frac{1}{q}}%
\end{array}%
\right\}
\end{eqnarray*}%
where we use the facts that%
\begin{equation*}
\int_{0}^{\frac{b-x}{b-a}}tdt=\frac{1}{2}\left( \frac{b-x}{b-a}\right) ^{2},
\end{equation*}

\begin{eqnarray*}
&&\int_{0}^{\frac{b-x}{b-a}}t\left\vert f^{\prime }(ta+(1-t)b)\right\vert
^{q}dt \\
&\leq &\frac{1}{3}\left( \frac{b-x}{b-a}\right) ^{3}\left\vert f^{\prime
}(a)\right\vert ^{q}+m\frac{(b-x)^{2}(b-3a+2x)}{6(b-a)^{3}}\left\vert
f^{\prime }\left( \frac{b}{m}\right) \right\vert ^{q},
\end{eqnarray*}

\begin{equation*}
\int_{\frac{b-x}{b-a}}^{1}(1-t)dt=\frac{1}{2}\left( \frac{x-a}{b-a}\right)
^{2},
\end{equation*}

\begin{eqnarray*}
&&\int_{\frac{b-x}{b-a}}^{1}(1-t)\left\vert f^{\prime
}(ta+(1-t)b)\right\vert ^{q}dt \\
&\leq &\left[ \frac{1}{6}+\frac{(b-x)^{2}(3a-2x-b)}{6(b-a)^{3}}\right]
\left\vert f^{\prime }(a)\right\vert ^{q}+m\frac{1}{3}\left( \frac{x-a}{b-a}%
\right) ^{3}\left\vert f^{\prime }\left( \frac{b}{m}\right) \right\vert ^{q}.
\end{eqnarray*}%
The proof is completed.
\end{proof}

\begin{remark}
\label{r.2.5} Suppose that all the assumptions of Theorem \ref{t.2.3} are
satisfied. If we choose $x=\frac{a+b}{2},$ we obtain
\end{remark}

\begin{equation*}
\left\vert f\left( \frac{a+b}{2}\right) -\frac{1}{b-a}\int_{a}^{b}f(u)du%
\right\vert \leq (b-a)\left( \frac{3^{1-\frac{1}{q}}}{8}\right) \left(
\left\vert f^{\prime }(a)\right\vert +m^{\frac{1}{q}}\left\vert f^{\prime
}\left( \frac{b}{m}\right) \right\vert \right)
\end{equation*}%
which is (\ref{h.1.2}).

\section{APPLICATIONS TO SPECIAL MEANS}

Let us recall the following means for two positive numbers.

$\left( AM\right) $ \textit{The Arithmetic mean}

\begin{equation*}
A=A(a,b)=\frac{a+b}{2};\text{ }a,b>0,
\end{equation*}

$\left( p-LM\right) $ \textit{The p-Logarithmic mean}%
\begin{equation*}
L_{p}=L_{p}(a,b)=\left\{ 
\begin{array}{cc}
a & if\text{ }a=b \\ 
\left[ \frac{b^{p+1}-a^{p+1}}{(p+1)(b-a)}\right] ^{\frac{1}{p}} & if\text{ }%
a\neq b%
\end{array}%
\right. ;\text{ }a,b>0,
\end{equation*}

$\left( IM\right) $ \textit{The Identric mean }%
\begin{equation*}
I=I(a,b)=\left\{ 
\begin{array}{cc}
a & if\text{ }a=b \\ 
\frac{1}{e}\left( \frac{b^{b}}{^{a^{a}}}\right) ^{\frac{1}{b-a}} & if\text{ }%
a\neq b%
\end{array}%
\right. ;\text{ }a,b>0.
\end{equation*}%
The following propositions hold:

\begin{proposition}
\label{p.3.1} Let $a,b\in \left[ 0,\infty \right) $, and $a<b$, $n\geq 2$
with $m\in \left( 0,1\right] $. Then we have
\end{proposition}

\begin{equation*}
\left\vert A^{n}(a,b)-L_{n}^{n}(a,b)\right\vert \leq n\frac{b-a}{8}\min
\left\{ 2A\left( a^{n-1},m\left( \frac{b}{m}\right) ^{^{n-1}}\right)
,2A\left( \left( b\right) ^{n-1},m\left( \frac{a}{m}\right) ^{^{n-1}}\right)
\right\} .
\end{equation*}

\begin{proof}
The proof follows by Remark \ref{r.2.1} on choosing $f:\left[ 0,\infty
\right) \rightarrow \left[ 0,\infty \right) $, $f(x)=x^{n},$ $n\in 
%TCIMACRO{\U{2124} }%
%BeginExpansion
\mathbb{Z}
%EndExpansion
,n\geq 2$ which is $m-$convex on $\left[ 0,\infty \right) .$
\end{proof}

\begin{proposition}
\label{p.3.2} Let $a,b\in \left[ 0,\infty \right) $, and $a<b$, with $m\in
\left( 0,1\right] $. Then we have
\end{proposition}

\begin{equation*}
\left\vert \ln \frac{I(a+1,b+1)}{A(a,b)+1}\right\vert \leq \frac{b-a}{4}%
\left( \eta _{1}^{\frac{1}{q}}+\eta _{2}^{\frac{1}{q}}\right) ,
\end{equation*}%
where

\begin{eqnarray*}
\eta _{1}^{\frac{1}{q}} &=&\min \left\{ \frac{\left( \frac{1}{b+1}\right)
^{q}+m\left( \frac{2m}{a+b+2m}\right) ^{q}}{2},\frac{\left( \frac{2}{a+b+2}%
\right) ^{q}+m\left( \frac{m}{b+m}\right) ^{q}}{2}\right\} , \\
\eta _{2}^{\frac{1}{q}} &=&\min \left\{ \frac{\left( \frac{1}{a+1}\right)
^{q}+m\left( \frac{2m}{a+b+2m}\right) ^{q}}{2},\frac{\left( \frac{2}{a+b+2}%
\right) ^{q}+m\left( \frac{m}{a+m}\right) ^{q}}{2}\right\} .
\end{eqnarray*}

\begin{proof}
The proof follows by Corollary \ref{c.2.2} on choosing $f:\left[ 0,\infty
\right) \rightarrow (-\infty ,0]$, $f(x)=-\ln (x+1)$ which is $m-$convex on $%
[0,\infty )$, $p>1.$
\end{proof}

\section{APPLICATIONS TO THE MIDPOINT FORMULA FOR $1-CONVEX$ FUNCTIONS}

Let $d$ be a division $a=x_{0}<x_{1}<...<x_{n-1}<x_{n}=b$ of the interval $%
\left[ a,b\right] $ and consider the quadrature formula

\begin{equation}
\int_{a}^{b}f(x)dx=M(f,d)+E(f,d),  \label{h.4.1}
\end{equation}%
where

\begin{equation*}
M(f,d)=\sum_{i=1}^{n-1}\left( x_{i+1}-x_{i}\right) f\left( \frac{%
x_{i+1}+x_{i}}{2}\right)
\end{equation*}%
is the midpoint formula and $E(f,d)$ denotes the associated approximation
error (see \cite{CPJP}).

Here, we obtain some error estimates for the midpoint formula.

\begin{proposition}
\label{p.4.1} Let $I$ be an open real interval such that $[0,\infty )\subset
I$. Let $f:I\rightarrow
%TCIMACRO{\U{211d} }%
%BeginExpansion
\mathbb{R}
%EndExpansion
$ be a differentiable function on $I$ such that $f^{\prime }\in L([a,b])$, where $0\leq a<b<\infty $ . If $\left\vert f^{\prime }\right\vert ^{q}$ is $%
1-$convex on $[a,b]$ for some fixed $m\in (0,1]$ and $p>1,$ $\frac{1}{p}+%
\frac{1}{q}=1$, then in (\ref{h.4.1}), for every division $d$ of $\left[ a,b%
\right] $, the midpoint error satisfies
\end{proposition}

\begin{equation*}
\left\vert E(f,d)\right\vert \leq \frac{1}{4}\sum_{i=0}^{n-1}\left(
x_{i+1}-x_{i}\right) ^{2}\left( \mu _{1}^{\frac{1}{q}}+\mu _{2}^{\frac{1}{q}%
}\right) ,
\end{equation*}%
where

\begin{eqnarray*}
\mu _{1} &=&\min \left\{ \frac{\left\vert f^{\prime }(x_{i})\right\vert
^{q}+\left\vert f^{\prime }(\frac{x_{i}+x_{i+1}}{2})\right\vert ^{q}}{2},%
\frac{\left\vert f^{\prime }(\frac{x_{i}+x_{i+1}}{2})\right\vert
^{q}+\left\vert f^{\prime }(x_{i})\right\vert ^{q}}{2}\right\} =\frac{%
\left\vert f^{\prime }(\frac{x_{i}+x_{i+1}}{2})\right\vert ^{q}+\left\vert
f^{\prime }(x_{i})\right\vert ^{q}}{2}, \\
\mu _{2} &=&\min \left\{ \frac{\left\vert f^{\prime }(x_{i+1})\right\vert
^{q}+\left\vert f^{\prime }(\frac{x_{i}+x_{i+1}}{2})\right\vert ^{q}}{2},%
\frac{\left\vert f^{\prime }(\frac{x_{i}+x_{i+1}}{2})\right\vert
^{q}+\left\vert f^{\prime }(x_{i+1})\right\vert ^{q}}{2}\right\} =\frac{%
\left\vert f^{\prime }(\frac{x_{i}+x_{i+1}}{2})\right\vert ^{q}+\left\vert
f^{\prime }(x_{i+1})\right\vert ^{q}}{2}.
\end{eqnarray*}

\begin{proof}
On applying Corollary \ref{c.2.2} with $m=1$  on the subinterval $\left[
x_{i},x_{i+1}\right] $ $(i=0,1,2,...,n-1)$ of the division, we have

\begin{equation*}
\left\vert f\left( \frac{x_{i+1}+x_{i}}{2}\right) -\frac{1}{x_{i+1}-x_{i}}%
\int_{x_{i}}^{x_{i+1}}f(x)dx\right\vert \leq \frac{x_{i+1}-x_{i}}{4}\left(
\mu _{1}^{\frac{1}{q}}+\mu _{2}^{\frac{1}{q}}\right) ,
\end{equation*}%
where%
\begin{eqnarray*}
\mu _{1} &=&\frac{\left\vert f^{\prime }(\frac{x_{i}+x_{i+1}}{2})\right\vert
^{q}+\left\vert f^{\prime }(x_{i})\right\vert ^{q}}{2} \\
\mu _{2} &=&\frac{\left\vert f^{\prime }(\frac{x_{i}+x_{i+1}}{2})\right\vert
^{q}+\left\vert f^{\prime }(x_{i+1})\right\vert ^{q}}{2}.
\end{eqnarray*}%
Hence, in (\ref{h.4.1}) we have

\begin{eqnarray*}
&&\left\vert \int_{a}^{b}f(x)dx-M(f,d)\right\vert \\
&=&\left\vert \sum_{i=0}^{n-1}\left[ \int_{x_{i}}^{x_{i+1}}f(x)dx-\left(
x_{i+1}-x_{i}\right) f\left( \frac{x_{i+1}-x_{i}}{2}\right) \right]
\right\vert \\
&\leq &\sum_{i=0}^{n-1}\left\vert \int_{x_{i}}^{x_{i+1}}f(x)dx-\left(
x_{i+1}-x_{i}\right) f\left( \frac{x_{i+1}-x_{i}}{2}\right) \right\vert \\
&\leq &\frac{1}{4}\sum_{i=0}^{n-1}\left( x_{i+1}-x_{i}\right) ^{2}\left( \mu
_{1}^{\frac{1}{q}}+\mu _{2}^{\frac{1}{q}}\right) ,
\end{eqnarray*}%
which completes the proof.
\end{proof}

\begin{proposition}
\label{p.4.2} Let $I$ be an open real interval such that $[0,\infty )\subset
I$. Let $f:I\rightarrow
%TCIMACRO{\U{211d} }%
%BeginExpansion
\mathbb{R}
%EndExpansion
$ be a differentiable function on $I$ such that $f^{\prime }\in L([a,b])$,
where $0\leq a<b<\infty $. If $\left\vert f^{\prime }\right\vert ^{q}$ is $%
1- $convex on $[a,b]$ for some fixed $m\in (0,1]$ and $q\in \lbrack 1,\infty
)$, $x\in \left[ a,b\right] $, then in (\ref{h.4.1}), for every division $d$
of $\left[ a,b\right] $, the midpoint error satisfies
\end{proposition}

\begin{equation*}
\left\vert E(f,d)\right\vert \leq \left( \frac{3^{1-\frac{1}{q}}}{8}\right)
\sum_{i=0}^{n-1}\left( x_{i+1}-x_{i}\right) ^{2}\left( \left\vert f^{\prime
}(x_{i})\right\vert +\left\vert f^{\prime }\left( x_{i+1}\right) \right\vert
\right) .
\end{equation*}

\begin{proof}
The proof is similar to that of Proposition \ref{p.4.1} and using Remark \ref%
{r.2.5} with $m=1$.
\end{proof}

\end{document}